\documentclass[12pt,reqno]{amsart}
\usepackage{amsmath,amssymb,amsfonts,amscd,latexsym,amsthm,mathrsfs,verbatim,comment,cite}
\usepackage{hyperref}
\let\eps\varepsilon
\let\lf\lfloor
\let\rf\rfloor
\renewcommand{\d}{{\mathrm d}}
\newcommand{\ba}{\boldsymbol{a}}
\newcommand{\bb}{\boldsymbol{b}}
\newcommand{\qbin}[2]{\genfrac{[}{]}{0pt}{}{#1}{#2}}
\newtheorem{theorem}{Theorem}
\newtheorem{lemma}{Lemma}
\begin{document}

\title{Congruences for $q$-binomial coefficients}

\date{23 January 2019. \emph{Revised}: 1 April 2019}

\author{Wadim Zudilin}
\address{Department of Mathematics, IMAPP, Radboud University, PO Box 9010, 6500~GL Nijmegen, Netherlands}
\email{w.zudilin@math.ru.nl}

\dedicatory{To George Andrews, with warm $q$-wishes and well-looking $q$-congruences}


\begin{abstract}
We discuss $q$-analogues of the classical congruence $\binom{ap}{bp}\equiv\binom{a}{b}\pmod{p^3}$, valid for primes $p>3$, as well as its generalisations. In particular, we prove related congruences for ($q$-analogues of) integral factorial ratios.
\end{abstract}

\maketitle

\section{Introduction}
\label{intro}

For $a$ non-negative integer, a standard $q$-environment includes
$q$-numbers $[a]=[a]_q=(1-q^a)/(1-q)\in\mathbb Z[q]$, $q$-factorials $[a]!=[1][2]\dotsb[a]\in\mathbb Z[q]$
and $q$-binomial coefficients
$$
\qbin{a}{b}={\qbin{a}{b}}_q=\frac{[a]!}{[b]!\,[a-b]!}\in\mathbb Z[q], \quad\text{where}\; b=0,1,\dots,a.
$$
One also adopts the cyclotomic polynomials
$$
\Phi_n(q)=\prod_{\substack{j=1\\(j,n)=1}}^n(q-e^{2\pi ij/n})\in\mathbb Z[q]
$$
as $q$-analogues of prime numbers, because these are the only factors of the $q$-numbers which are irreducible over~$\mathbb Q$.

Arithmetically significant relations often possess \emph{several} $q$-analogues.
While looking for $q$-extensions of the classical (Wolstenholme--Ljunggren) congruence
\begin{equation}
\binom{ap}{bp}\equiv\binom{a}{b}\pmod{p^3} \quad\text{for any prime}\; p>3,
\label{classics}
\end{equation}
more precisely, at a `$q$-microscope setup' (when $q$-congruences for truncated hypergeometric sums are read off from
the asymptotics of their non-terminating versions, usually equipped with extra parameters, at roots of unity\,---\,see~\cite{GZ19}) for Straub's $q$-congruence \cite{Str11}, \cite[Theorem 2.2]{Str18},
\begin{equation}
{\qbin{an}{bn}}_q\equiv{\qbin ab}_{q^{n^2}}-b(a-b)\binom ab\frac{n^2-1}{24}\,(q^n-1)^2\pmod{\Phi_n(q)^3},
\label{straub}
\end{equation}
this author accidentally arrived at
\begin{equation}
{\qbin{an}{bn}}\sigma_n^bq^{\binom{bn}2}
\equiv\binom{a-1}b+\binom{a-1}{a-b}\sigma_n^aq^{\binom{an}2}\pmod{\Phi_n(q)^2},
\label{cong}
\end{equation}
where the notation
$$
\sigma_n=(-1)^{n-1}
$$
is implemented.
Notice that the expression on the right-hand side is a sum of two $q$-monomials.
The $q$-congruence \eqref{cong} may be compared with another $q$-extension of~\eqref{classics},
\begin{equation}
{\qbin{an}{bn}}_q\equiv\sigma_n^{b(a-b)}q^{b(a-b)\binom n2}{\qbin{a}{b}}_{q^n}\pmod{\Phi_n(q)^2}
\label{andrews}
\end{equation}
for any $n>1$. This is given by Andrews in \cite{And99} for primes $n=p>3$ only;
though proved modulo $\Phi_p(q)^2$, a complimentary result from \cite{And99} demonstrates
that \eqref{classics} in its full modulo $p^3$ strength can be derived from~\eqref{andrews}.
More directly, Pan \cite{Pan13} shows that \eqref{andrews} can be generalised further to
\begin{equation}
{\qbin{an}{bn}}_q\equiv\sigma_n^{b(a-b)}q^{b(a-b)\binom n2}{\qbin ab}_{q^n}+ab(a-b)\binom ab\frac{n^2-1}{24}\,(q^n-1)^2\pmod{\Phi_n(q)^3}.
\label{pan}
\end{equation}
It is worth mentioning that the transition from $\Phi_n(q)^2$ to $\Phi_n(q)^3$ (or, from~$p^2$ to~$p^3$)
is significant because the former has a simple combinatorial proof (resulting from the ($q$-)Chu--Vandermonde identity)
whereas no combinatorial proof is known for the latter.

Since $q^{\binom{bn}2}\sim\sigma_n^b$ as $q\to\zeta$, a primitive $n$-th root of unity, the congruence \eqref{cong} is seen to be an extension of the trivial ($q$-Lucas) congruence
\begin{equation*}
{\qbin{an}{bn}}\equiv\binom ab=\binom{a-1}b+\binom{a-1}{a-b}\pmod{\Phi_n(q)}.
\end{equation*}

The principal goal of this note is to provide a modulo $\Phi_n(q)^4$ extension of \eqref{cong}
(see Lemma~\ref{lem1} below)
as well as to use the result for extending the congruences \eqref{straub} and~\eqref{pan}.
In this way, our theorems provide two $q$-extensions of the congruence
\begin{equation*}
\binom{ap}{bp}\equiv\binom{a}{b}+ab(a-b)\binom{a}{b}p\sum_{k=1}^{p-1}\frac1k\pmod{p^4}
\quad\text{for prime}\; p>3.
\end{equation*}
The latter can be continued further to higher powers of primes \cite{Mes11}, and our `mechanical' approach here
suggests that one may try\,---\,with a lot of effort!\,---\,to deduce corresponding $q$-analogues.

\begin{theorem}
\label{th1}
The congruence
\begin{align}
{\qbin{an}{bn}}_q
&\equiv{\qbin{a}{b}}_{q^{n^2}}
-b(a-b)\binom ab(q^n-1)\biggl(a\sum_{k=1}^{n-1}\frac{q^k}{1-q^k}+\frac{a(n-1)}2
\nonumber\\ &\quad
+\frac{(a+1)(n^2-1)}{24}(q^n-1)
+\frac{(b(a-b)n-a-2)(n^2-1)}{48}\,(q^n-1)^2\biggr)
\label{straub-e}
\end{align}
holds modulo $\Phi_n(q)^4$ for any $n>1$.
\end{theorem}

\begin{theorem}
\label{th2}
For any $n>1$, we have the congruence
\begin{align}
{\qbin{an}{bn}}_q
&\equiv\sigma_n^{b(a-b)}q^{b(a-b)\binom n2}{\qbin{a}{b}}_{q^n}
-ab(a-b)\binom ab(q^n-1)\biggl(\sum_{k=1}^{n-1}\frac{q^k}{1-q^k}+\frac{n-1}2
\nonumber\\ &\quad
-\frac{(b(a-b)n-1)(n^2-1)}{48}\,(q^n-1)^2\biggr)
\pmod{\Phi_n(q)^4}.
\label{pan-e}
\end{align}
\end{theorem}

We point out that 
a congruence $A_1(q)\equiv A_2(q)\pmod{P(q)}$
for \emph{rational functions} $A_1(q),A_2(q)\in\mathbb Q(q)$ and a polynomial $P(q)\in\mathbb Q[q]$
is understood as follows: the polynomial $P(q)$ is relatively prime with the denominators of $A_1(q)$ and $A_2(q)$, and
$P(q)$ divides the numerator $A(q)$ of the difference $A_1(q)-A_2(q)$.
The latter is equivalent to the condition that for each zero $\alpha\in\mathbb C$ of $P(q)$ of multiplicity $k$,
the polynomial $(q-\alpha)^k$ divides $A(q)$ in $\mathbb C[q]$; in other words, $A_1(q)-A_2(q)=O\bigl((q-\alpha)^k\bigr)$
as $q\to\alpha$. This latter interpretation underlies our argument in proving the results.
For example, the congruence \eqref{cong} can be established by verifying that
\begin{equation}
{\qbin{an}{bn}}_q(1-\eps)^{\binom{bn}2}
=\binom{a-1}b+\binom{a-1}{a-b}\sigma_n^a(1-\eps)^{\binom{an}2}+O(\eps^2)
\quad\text{as}\; \eps\to0^+,
\label{cong-asymp}
\end{equation}
when $q=\zeta(1-\eps)$ and $\zeta$ is any primitive $n$-th root of unity.

Our approach goes in line with \cite{GZ19} and shares similarities with the one developed by Gorodetsky in \cite{Gor18},
who reads off the asymptotic information of binomial sums at roots of unity through $q$-Gauss congruences.
It does not seem straightforward to us but Gorodetsky's method may be capable of proving Theorems~\ref{th1} and~\ref{th2}.
Furthermore, the part \cite[Sect.~2.3]{Gor18} contains a survey on $q$-analogues of \eqref{classics}.

After proving an asymptotical expansion for $q$-binomial coefficients at roots of unity in Section~\ref{roots}
(essentially, the $O(\eps^4)$-extension of \eqref{cong-asymp}), we perform a similar asymptotic analysis
for $q$-harmonic sums in Section~\ref{harm}. The information gathered is then applied in Section~\ref{proofs}
to proving Theorems~\ref{th1} and~\ref{th2}. Finally, in Section~\ref{ratios} we generalise
the congruences \eqref{straub} and \eqref{pan} in a different direction, to integral factorial ratios.

\section{Expansions of $q$-binomials at roots of unity}
\label{roots}

This section is exclusively devoted to an asymptotical result, which forms the grounds of our later arithmetic analysis.
We moderate its proof by highlighting principal ingredients (and difficulties) of derivation and leaving some technical details to the reader.

\begin{lemma}
\label{lem1}
Let $\zeta$ be a primitive $n$-th root of unity. Then, as $q=\zeta(1-\eps)\to\zeta$ radially,
\begin{align}
&
{\qbin{an}{bn}}_{q}\sigma_n^bq^{\binom{bn}2}
-\binom{a-1}b-\binom{a-1}{a-b}\sigma_n^aq^{\binom{an}2}
\nonumber\\ &\quad
=b(a-b)\binom ab\bigl(-\eps^2n^2\rho_0(a,n)
+\eps^3n^2\rho_1(a,b,n)
+\eps^3\,anS_{n-1}(\zeta)\bigr)
+O(\eps^4),
\label{asymp1}
\end{align}
where
\begin{align*}
\rho_0(a,n)&=\frac{3(an-1)^2-an^2-1}{24},
\\
\rho_1(a,b,n)&=\frac{abn^2(an-1)(an-n-2)+(an+2)(an-1)^2(an-3)+an^2+a+2}{48}
\end{align*}
and
$$
S_{n-1}(q)=\frac12\sum_{k=1}^{n-1}\frac{kq^k((k+1)q^k+k-1)}{(1-q^k)^3}.
$$
\end{lemma}

\begin{proof}
It follows from the $q$-binomial theorem \cite[Chap.~10]{AAR99} that
\begin{equation}
(x;q)_N=\sum_{k=0}^N{\qbin Nk}_q(-x)^kq^{\binom k2}.
\label{qbin}
\end{equation}
Taking $N=an$, for a primitive $n$-th root of unity $\zeta=\zeta_n$, we have
\begin{equation}
\frac1n\sum_{j=1}^n(\zeta^jx;q)_{an}
=\sum_{\substack{k=0\\n\mid k}}^{an}{\qbin{an}k}(-x)^kq^{k(k-1)/2}
=\sum_{b=0}^a{\qbin{an}{bn}}(-x)^{bn}q^{bn(bn-1)/2}.
\label{eq1}
\end{equation}

When $q=\zeta(1-\eps)$, we get $\d/\d\eps=-\zeta\,(\d/\d q)$.
If
$$
f(q)=(x;q)_{an} \quad\text{and}\quad
g(q)=\frac{\d}{\d q}\log f(q)=-\sum_{\ell=1}^{an-1}\frac{\ell q^{\ell-1}x}{1-q^\ell x},
$$
then $f(q)|_{\eps=0}=(1-x^n)^a$ and
$$
\frac{\d f}{\d q}=fg, \quad
\frac{\d^2f}{\d q^2}=f\biggl(g^2+\frac{\d g}{\d q}\biggr), \quad
\frac{\d^3f}{\d q^3}=f\biggl(g^3+3g\frac{\d g}{\d q}+\frac{\d^2g}{\d q^2}\biggr).
$$
In particular,
\begin{align*}
\frac{\d f}{\d\eps}\bigg|_{\eps=0}
&=(1-x^n)^a\sum_{\ell=1}^{an-1}\frac{\ell\zeta^\ell x}{1-\zeta^\ell x},
\\
\frac{\d^2f}{\d\eps^2}\bigg|_{\eps=0}
&=(1-x^n)^a\biggl(\biggl(\sum_{\ell=1}^{an-1}\frac{\ell\zeta^\ell x}{1-\zeta^\ell x}\biggr)^2
-\sum_{\ell=1}^{an-1}\biggl(\frac{\ell^2\zeta^{2\ell}x^2}{(1-\zeta^\ell x)^2}+\frac{\ell(\ell-1)\zeta^\ell x}{1-\zeta^\ell x}\biggr)\biggr)
\\ \intertext{and}
\frac{\d^3f}{\d\eps^3}\bigg|_{\eps=0}
&=(1-x^n)^a\biggl(\biggl(\sum_{\ell=1}^{an-1}\frac{\ell\zeta^\ell x}{1-\zeta^\ell x}\biggr)^3
\\ &\qquad
-3\sum_{\ell=1}^{an-1}\frac{\ell\zeta^\ell x}{1-\zeta^\ell x}
\sum_{\ell=1}^{an-1}\biggl(\frac{\ell^2\zeta^{2\ell}x^2}{(1-\zeta^\ell x)^2}+\frac{\ell(\ell-1)\zeta^\ell x}{1-\zeta^\ell x}\biggr)
\\ &\qquad
+\sum_{\ell=1}^{an-1}\biggl(\frac{2\ell^3\zeta^{3\ell}x^3}{(1-\zeta^\ell x)^3}+\frac{3\ell^2(\ell-1)\zeta^{2\ell}x^2}{(1-\zeta^\ell x)^2}+\frac{\ell(\ell-1)(\ell-2)\zeta^\ell x}{1-\zeta^\ell x}\biggr)\biggr).
\end{align*}
Now observe the following summation formulae:
\begin{align*}
\frac1n\sum_{j=1}^n\frac{x}{1-x}\bigg|_{x\mapsto\zeta^jx}
&=\frac{x^n}{1-x^n},
\\
\frac1n\sum_{j=1}^n\biggl(\frac{x}{1-x}\biggr)^2\bigg|_{x\mapsto\zeta^jx}
&=\frac{nx^n}{(1-x^n)^2}-\frac{x^n}{1-x^n},
\displaybreak[2]\\
\frac1n\sum_{j=1}^n\frac{x}{1-x}\,\frac{\zeta^kx}{1-\zeta^kx}\bigg|_{x\mapsto\zeta^jx}
&=-\frac{x^n}{1-x^n} \quad\text{for}\; k\not\equiv0\pmod n,
\displaybreak[2]\\
\frac1n\sum_{j=1}^n\biggl(\frac{x}{1-x}\biggr)^3\bigg|_{x\mapsto\zeta^jx}
&=\frac{n^2x^n(1+x^n)}{2(1-x^n)^3}-\frac{3nx^n}{2(1-x^n)^2}+\frac{x^n}{1-x^n},
\displaybreak[2]\\
\frac1n\sum_{j=1}^n\frac{x}{1-x}\,\frac{\zeta^kx}{1-\zeta^kx}\,\frac{\zeta^\ell x}{1-\zeta^\ell x}\bigg|_{x\mapsto\zeta^jx}
&=\frac{x^n}{1-x^n} \quad\text{for}\; k\not\equiv0,\; \ell\not\equiv0,\; k\not\equiv\ell \pmod n,
\\ \intertext{and}
\frac1n\sum_{j=1}^n\biggl(\frac{x}{1-x}\biggr)^2\frac{\zeta^kx}{1-\zeta^kx}\bigg|_{x\mapsto\zeta^jx}
&=\frac{nx^n}{(1-x^n)^2}\,\frac{\zeta^k}{1-\zeta^k}+\frac{x^n}{1-x^n}
\\ &\qquad\qquad\qquad\text{for}\; k\not\equiv0\pmod n.
\end{align*}
Implementing this information into \eqref{eq1} we obtain
\begin{align*}
&
\sum_{b=0}^a{\qbin{an}{bn}}(-x)^{bn}q^{bn(bn-1)/2}\bigg|_{q=\zeta(1-\eps)}
=(1-x^n)^a\Biggl(1+\eps\,\frac{x^n}{1-x^n}\sum_{\ell=1}^{an-1}\ell
\\ &\quad
-\frac{\eps^2}2\,\frac{x^n}{1-x^n}\biggl(\sum_{\ell=1}^{an-1}\ell\biggr)^2
+\frac{\eps^2}2\,\frac{nx^n}{(1-x^n)^2}\sum_{\substack{\ell_1,\ell_2=1\\ \ell_1\equiv\ell_2\pmod n}}^{an-1}\ell_1\ell_2
\\ &\quad
-\frac{\eps^2}2\biggl(\frac{nx^n}{(1-x^n)^2}-\frac{x^n}{1-x^n}\biggr)\sum_{\ell=1}^{an-1}\ell^2
-\frac{\eps^2}2\,\frac{x^n}{1-x^n}\sum_{\ell=1}^{an-1}\ell(\ell-1)
\\ &\quad
+\dots
+\frac{\eps^3}2\,\frac{nx^n}{(1-x^n)^2}\sum_{\substack{\ell_1,\ell_2,\ell_3=1\\ \ell_1\equiv\ell_2\not\equiv\ell_3\pmod n}}^{an-1}\ell_1\ell_2\ell_3\,\frac{\zeta^{\ell_3-\ell_1}}{1-\zeta^{\ell_3-\ell_1}}
+\dotsb
\displaybreak[2]\\ &\quad
+\dots
-\frac{\eps^3}2\,\frac{nx^n}{(1-x^n)^2}\sum_{\substack{\ell_1,\ell_2=1\\ \ell_1\not\equiv\ell_2\pmod n}}^{an-1}
\ell_1^2\ell_2\,\frac{\zeta^{\ell_2-\ell_1}}{1-\zeta^{\ell_2-\ell_1}}
+\dotsb\Biggr)
+O(\eps^4),
\end{align*}
where we intentionally omit all \emph{ordinary} $\eps^3$-terms\,---\,those that sum up to polynomials in $a$ and $n$ multiplied by powers of $x^n/(1-x^n)$, like the ones appearing as $\eps$- and $\eps^2$-terms.
The exceptional $\eps^3$-summands are computed separately:
\begin{multline*}
\sum_{\substack{\ell_1,\ell_2,\ell_3=1\\ \ell_1\equiv\ell_2\not\equiv\ell_3\pmod n}}^{an-1}\ell_1\ell_2\ell_3\,\frac{\zeta^{\ell_3-\ell_1}}{1-\zeta^{\ell_3-\ell_1}}
=-a^3\sum_{k=1}^{n-1}\frac{k\zeta^k((k+1)\zeta^k+k-1)}{(1-\zeta^k)^3}
\\
-\frac{a^3n(n-1)(3an(an-1)(an-2)+n^2(an-2a-1)-2)}{48}
\end{multline*}
and
\begin{multline*}
\sum_{\substack{\ell_1,\ell_2=1\\ \ell_1\not\equiv\ell_2\pmod n}}^{an-1}
\ell_1^2\ell_2\,\frac{\zeta^{\ell_2-\ell_1}}{1-\zeta^{\ell_2-\ell_1}}
=-a^2\sum_{k=1}^{n-1}\frac{k\zeta^k((k+1)\zeta^k+k-1)}{(1-\zeta^k)^3}
\\
-\frac{a^2n(n-1)(an(an-1)(2an-3)-an^2-1)}{24}.
\end{multline*}
The finale of our argument is comparison of the coefficients of powers of $x^n$ on both sides of the relation obtained;
this way we arrive at the asymptotics in~\eqref{asymp1}.
\end{proof}

\section{A $q$-harmonic sum}
\label{harm}

Again, the notation $\zeta$ is reserved for a primitive $n$-th root of unity.
For the sum
$$
H_{n-1}(q)=\sum_{k=1}^{n-1}\frac{q^k}{1-q^k},
$$
we have
\begin{align*}
\frac{\d H_{n-1}}{\d q}
&=\sum_{k=1}^{n-1}\frac{kq^{k-1}}{(1-q^k)^2},
\\
\frac{\d^2H_{n-1}}{\d q^2}
&=\sum_{k=1}^{n-1}\frac{kq^{k-2}((k+1)q^k+k-1)}{(1-q^k)^3}=2q^{-2}S_{n-1}(q),
\end{align*}
where $S_{n-1}(q)$ is defined in Lemma~\ref{lem1}.
It follows that, for $q=\zeta(1-\eps)$,
\begin{align}
H_{n-1}(q)
&=\sum_{k=1}^{n-1}\frac{\zeta^k}{1-\zeta^k}
-\eps\sum_{k=1}^{n-1}\frac{k\zeta^k}{(1-\zeta^k)^2}
+\frac{\eps^2}2\sum_{k=1}^{n-1}\frac{k\zeta^k((k+1)\zeta^k+k-1)}{(1-\zeta^k)^3}
+O(\eps^3)
\nonumber\\
&=-\frac{n-1}2+\frac{(n^2-1)n}{24}\,\eps+S_{n-1}(\zeta)\eps^2
+O(\eps^3)
\displaybreak[2] \nonumber\\
&=-\frac{n-1}2-\frac{n^2-1}{24}\,(q^n-1)+\frac{(n-1)(n^2-1)}{48n}(q^n-1)^2
\nonumber\\ &\qquad
+\frac1{n^2}S_{n-1}(\zeta)(q^n-1)^2
+O(\eps^3)
\label{eq3}
\end{align}
as $\eps\to0$, where we use
$$
\eps=-\frac1n\,(q^n-1)+\frac{n-1}{2n^2}\,(q^n-1)^2+O(\eps^3)
\quad\text{as}\; \eps\to0.
$$
The latter asymptotics implies that
\begin{align*}
\sum_{k=1}^{n-1}\frac{q^k}{1-q^k}
&\equiv-\frac{n-1}2-\frac{n^2-1}{24}\,(q^n-1)+\frac{(n-1)(n^2-1)}{48n}(q^n-1)^2
\\ &\qquad
+\frac{(q^n-1)^2}{2n^2}\sum_{k=1}^{n-1}\frac{kq^k((k+1)q^k+k-1)}{(1-q^k)^3}
\pmod{\Phi_n(q)^3},
\end{align*}
which may be viewed as an extension of
$$
\sum_{k=1}^{n-1}\frac{q^k}{1-q^k}
\equiv-\frac{n-1}2-\frac{n^2-1}{24}\,(q^n-1)
\pmod{\Phi_n(q)^2}
$$
recorded, for example, in~\cite{Mes11}.

A different consequence of \eqref{eq3} is the following fact.

\begin{lemma}
\label{lem2}
The term $\eps^2S_{n-1}(\zeta)$ appearing in the expansion \eqref{asymp1} can be replaced with
$$
H_{n-1}(q)+\frac{n-1}2+\frac{n^2-1}{24}\,(q^n-1)-\frac{(n-1)(n^2-1)}{48n}(q^n-1)^2+O(\eps^3)
$$
when $q=\zeta(1-\eps)$ and $\eps\to0$.
\end{lemma}

\section{Proof of the theorems}
\label{proofs}

In order to prove Theorems \ref{th1} and \ref{th2} we need to produce `matching' asymptotics for
$$
{\qbin{a}{b}}_{q^{n^2}}
\quad\text{and}\quad
\sigma_n^{b(a-b)}q^{b(a-b)\binom n2}{\qbin{a}{b}}_{q^n},
$$
respectively. These happen to be easier than that from Lemma~\ref{lem1}
because $q^{n^2}=(1-\eps)^{n^2}$ and $q^n=(1-\eps)^n$ do not depend
on the choice of primitive $n$-th root of unity $\zeta$ when $q=\zeta(1-\eps)$.

\begin{lemma}
\label{lem3}
As $q=\zeta(1-\eps)\to\zeta$ radially,
\begin{align*}
&
{\qbin{a}{b}}_{q^{n^2}}\sigma_n^bq^{\binom{bn}2}
-\binom{a-1}b-\binom{a-1}{a-b}\sigma_n^aq^{\binom{an}2}
\\ &\quad
=b(a-b)\binom ab\bigl(-\eps^2n^2\hat\rho_0(a,n)
+\eps^3n^2\hat\rho_1(a,b,n)\bigr)
+O(\eps^4),
\end{align*}
where
\begin{align*}
\hat\rho_0(a,n)&=\frac{3(an-1)^2-(a+1)n^2}{24},
\\
\hat\rho_1(a,b,n)&=\frac{bn(an-1)((an-1)^2-(a+1)n^2)}{48}
\\ &\qquad
+\frac{an(an-1)^3-6(an-1)^2+2(a+1)n^2}{48}.
\end{align*}
\end{lemma}

\begin{proof}
For $N=a$ in \eqref{qbin}, take $x^nq^{\binom n2}$ and $q^{n^2}$ for $x$ and~$q$:
\begin{equation*}
(x^nq^{\binom n2};q^{n^2})_a=\sum_{b=0}^a{\qbin ab}_{q^{n^2}}\sigma_n^b(-x)^{bn}q^{\binom{bn}2}.
\end{equation*}
Then, for $q=\zeta(1-\eps)$, we write $y=\sigma_nx^n$ to obtain
\begin{align*}
(x^nq^{\binom n2};q^{n^2})_a
&=(y(1-\eps)^{\binom n2};(\eps)^{n^2})_a
=\prod_{\ell=0}^{a-1}\bigl(1-y(1-\eps)^{\ell n^2+\binom n2}\bigr)
\\
&=(1-y)^a\prod_{\ell=0}^{a-1}\biggl(1-\frac{y}{1-y}
\sum_{i=1}^{\ell n^2+\binom n2}\binom{\ell n^2+\binom n2}i(-\eps)^i\biggr).
\end{align*}
To conclude, we apply the same argument as in the proof of Lemma~\ref{lem1}.
\end{proof}

\begin{proof}[Proof of Theorem~\textup{\ref{th1}}]
Combining the expansions in Lemmas~\ref{lem1}--\ref{lem3} we find out that
\begin{align*}
&
{\qbin{an}{bn}}_{q}\sigma_n^bq^{\binom{bn}2}
-{\qbin{a}{b}}_{q^{n^2}}\sigma_n^bq^{\binom{bn}2}
=-b(a-b)\binom ab(q^n-1)\biggl(a\sum_{k=1}^{n-1}\frac{q^k}{1-q^k}+\frac{a(n-1)}2
\\ &\quad
+\frac{(a+1)(n^2-1)}{24}(q^n-1)
+\frac{(b(a-b)n-a-2)(n^2-1)}{48}\,(q^n-1)^2\biggr)+O(\eps^4)
\end{align*}
as $q=\zeta(1-\eps)\to\zeta$ radially. This means that the difference of both sides
is divisible by $(q-\zeta)^4$ for any $n$-th primitive root of unity $\zeta$, hence by $\Phi_n(q)^4$.
The latter property is equivalent to the congruence~\eqref{straub-e}.
\end{proof}

\begin{proof}[Proof of Theorem~\textup{\ref{th2}}]
We first use Lemma \ref{lem1} with $n=1$:
\begin{align*}
&
{\qbin{a}{b}}_{q}q^{\binom{b}2}
-\binom{a-1}b-\binom{a-1}{a-b}q^{\binom{a}2}
\\ &\quad
=b(a-b)\binom ab\bigl(-(1-q)^2\rho_0(a,1)+(1-q)^3\rho_1(a,b,1)\bigr)
+O\bigl((1-q)^4\bigr)
\end{align*}
as $q\to1$. Now, take $n>1$ arbitrary and apply this relation for $q$ replaced with $q^n$, where $q=\zeta(1-\eps)$, $0<\eps<1$ and $\zeta$ is a primitive $n$-th root of unity:
\begin{align*}
&
\sigma_n^{b(a-b)}q^{b(a-b)\binom n2}{\qbin{a}{b}}_{q^n}
\\ &\quad
=\binom{a-1}b(1-\eps)^{b(a-b)\binom n2-\binom b2n}
+\binom{a-1}{a-b}(1-\eps)^{b(a-b)\binom n2+\binom a2n-\binom b2n}
\\ &\quad\qquad
+b(a-b)\binom ab\bigl(-(1-(1-\eps)^n)^2\rho_0(a,1)+\eps^3n^3\rho_1(a,b,1)\bigr)
\\ &\quad\qquad\quad\times
(1-\eps)^{b(a-b)\binom n2-\binom b2n}
+O(\eps^4)
\displaybreak[2]\\ &\quad
=\binom{a-1}b(1-\eps)^{-\binom{bn}2+ab\binom n2}
+\binom{a-1}{a-b}(1-\eps)^{\binom{an}2-\binom{bn}2-a(a-b)\binom n2}
\\ &\quad\qquad
+b(a-b)\binom ab\bigl(-\eps^2n^2\rho_0(a,1)+\eps^3n^2\bigl((n-1)\rho_0(a,1)+n\rho_1(a,b,1)\bigr)\bigr)
\\ &\quad\qquad\quad\times
\biggl(1-\frac{((a-b)n-a+1)bn}2\,\eps+O(\eps^2)\biggr)
+O(\eps^4)
\end{align*}
as $\eps\to0$. At the same time, from Lemma~\ref{lem1} we have
\begin{align*}
{\qbin{an}{bn}}_q
&=\binom{a-1}b(1-\eps)^{-\binom{bn}2}+\binom{a-1}{a-b}(1-\eps)^{\binom{an}2-\binom{bn}2}
\\ &\qquad
+b(a-b)\binom ab\bigl(-\eps^2n^2\rho_0(a,n)
+\eps^3n^2\rho_1(a,b,n)
+\eps^3\,anS_{n-1}(\zeta)\bigr)
\\ &\qquad\quad\times
\biggl(1+\binom{bn}2\eps+O(\eps^2)\biggr)
+O(\eps^4)
\end{align*}
as $\eps\to0$. Using
$$
(1-\eps)^N=1-N\eps+\binom N2\eps^2-\binom N3\eps^3+O(\eps^4)
\quad\text{as}\; \eps\to0
$$
for $N=-\binom{bn}2+ab\binom n2$, $\binom{an}2-\binom{bn}2-a(a-b)\binom n2$, $-\binom{bn}2$ and $\binom{an}2-\binom{bn}2$ we deduce from the two expansions and Lemma~\ref{lem2} that, for $q=\zeta(1-\eps)$,
\begin{align*}
{\qbin{an}{bn}}_q
-\sigma_n^{b(a-b)}q^{b(a-b)\binom n2}{\qbin{a}{b}}_{q^n}
&=-ab(a-b)\binom ab(q^n-1)\biggl(\sum_{k=1}^{n-1}\frac{q^k}{1-q^k}+\frac{n-1}2
\\ &\qquad
-\frac{(b(a-b)n-1)(n^2-1)}{48}\,(q^n-1)^2\biggr)
+O(\eps^4)
\end{align*}
as $\eps\to0$. This implies the congruence in~\eqref{pan-e}.
\end{proof}

\section{$q$-rious congruences}
\label{ratios}

In this final part, we look at the binomial coefficients as particular instances
of integral ratios of factorials, also known as Chebyshev--Landau factorial ratios.
In the $q$-setting these are defined by
$$
D_n(q)=D_n(\ba,\bb;q)=\frac{[a_1n]!\dotsb[a_rn]!}{[b_1n]!\dotsb[b_sn]!},
$$
where $\ba=(a_1,\dots,a_r)$ and $\bb=(b_1,\dots,b_s)$ are positive integers satisfying
\begin{equation}
a_1+\dots+a_r=b_1+\dots+b_s
\label{cond1}
\end{equation}
and
\begin{equation}
\lf a_1x\rf+\dots+\lf a_rx\rf\ge\lf b_1x\rf+\dots+\lf b_sx\rf
\quad\text{for all}\; x>0
\label{cond2}
\end{equation}
(see, for example, \cite{WZ11}), $\lf\,\cdot\,\rf$ denotes the integer part of a number.
Then $D_n(q)\in\mathbb Z[q]$ are polynomials with values
$$
D_n(1)=\frac{(a_1n)!\dotsb(a_rn)!}{(b_1n)!\dotsb(b_sn)!}
$$
at $q=1$, and the congruences \eqref{straub} and \eqref{pan} generalise as follows.

\begin{theorem}
\label{th3}
In the notation
$$
c_i=c_i(\ba,\bb)=\binom{a_1}i+\dots+\binom{a_r}i-\binom{b_1}i-\dotsb-\binom{b_s}i
\quad\text{for}\; i=2,3,
$$
the congruences
\begin{align}
D_n(q)&\equiv D_1(q^{n^2})-D_1(1)\frac{c_2\,(n^2-1)}{24}\,(q^n-1)^2\pmod{\Phi_n(q)^3}
\label{straub-g}
\\ \intertext{and}
D_n(q)&\equiv \sigma_n^{c_2}q^{c_2\binom n2}D_1(q^n)+D_1(1)\frac{(c_2+c_3)\,(n^2-1)}{12}\,(q^n-1)^2\pmod{\Phi_n(q)^3}
\label{pan-g}
\end{align}
are valid for any $n\ge1$.
\end{theorem}

Observe that when $n=p>3$ and $q\to1$, one recovers from any of these two the congruences
$$
D_p(1)\equiv D_1(1)\pmod{p^3},
$$
of which \eqref{classics} is a special case.
Furthermore, it is tempting to expect that these two families of $q$-congruences may be generalised even further
in the spirit of Theorems~\ref{th1} and \ref{th2}, and that the polynomials $D_n(q)$ satisfy $q$-Gauss relations from~\cite{Gor18}.
We do not pursue this line here.

\begin{proof}[Proof of Theorem~\textup{\ref{th3}}]
Though the congruences \eqref{straub-g} and \eqref{pan-g} are between \emph{polynomials} rather than
rational functions, we prove the theorem without assumption~\eqref{cond2}: in other words, the congruences remain
true for the rational functions $D_n(q)$ provided that the balancing condition \eqref{cond1} (equivalently, $c_1(\ba,\bb)=1$ in the above notation for $c_i$) is satisfied. In turn, this more general statement follows from its validity for
particular cases
$$
D_n(q)=\frac{[an]!}{[bn]!\,[(a-b)n]!}
\quad\text{and}\quad
\tilde D_n(q)=\frac{[bn]!\,[(a-b)n]!}{[an]!}
$$
by induction (on $r+s$, say). Indeed, the inductive step exploits the property of both \eqref{straub-g} and \eqref{pan-g}
to imply the congruence for the product $D_n(\ba,\bb;q)D_n(\tilde{\ba},\tilde{\bb};q)$ whenever it is already known
for the individual factors; we leave this simple fact to the reader and only discuss its other appearance when dealing with $\tilde D_n(q)$ below.
Notice that $\qbin{an}{bn}\equiv\binom{an}{bn}\not\equiv0\pmod{\Phi_n(q)}$, so that
$\tilde D_n(q)={\qbin{an}{bn}}^{-1}$ is well defined modulo any power of~$\Phi_n(q)$.

For $D_n(q)=\qbin{an}{bn}$ we have $c_2=b(a-b)$ and $c_2+c_3=ab(a-b)/2$, hence \eqref{straub-g} and \eqref{pan-g}
follow from \eqref{straub} and \eqref{pan}, respectively.

Turning to $q=\zeta(1-\eps)$, where $0<\eps<1$ and $\zeta$ is a primitive $n$-th root of unity,
write the congruences \eqref{straub} and \eqref{pan} as the asymptotic relation
$$
\qbin{an}{bn}=B(q)+cB(1)\eps^2+O(\eps^3) \quad\text{as}\; \eps\to0,
$$
in which
\begin{gather*}
B(q)={\qbin ab}_{q^{n^2}}, \quad c=-\frac{b(a-b)n^2(n^2-1)}{24} \quad\text{for the case \eqref{straub}}
\\ \intertext{and}
B(q)=\sigma_n^{b(a-b)}q^{b(a-b)\binom n2}{\qbin ab}_{q^n}, \quad c=\frac{ab(a-b)n^2(n^2-1)}{24} \quad\text{for the case \eqref{pan}}.
\end{gather*}
Then
\begin{align*}
\tilde D_n(q)={\qbin{an}{bn}}^{-1}
&=B(q)^{-1}\,\bigl(1+cB(1)B(q)^{-1}\eps^2+O(\eps^3)\bigr)^{-1}
\displaybreak[2]\\
&=B(q)^{-1}-cB(1)B(q)^{-2}\eps^2+O(\eps^3)
\\
&=B(q)^{-1}-cB(1)^{-1}\eps^2+O(\eps^3),
\end{align*}
because we have $B(q)=B(1)+O(\eps)$ as $\eps\to0$ for our choices of $B(q)$.
The resulting expansion implies the truth of \eqref{straub-g} and \eqref{pan-g} for
$\tilde D_n(q)=D_n((b,a-b),(a);q)$ in view~of
$$
c_i((b,a-b),(a))=-c_i((a),(b,a-b))
\quad\text{for}\; i=2,3.
$$
As explained above, this also establishes the general case of \eqref{straub-g} and~\eqref{pan-g}.
\end{proof}

For related Lucas-type congruences satisfied by the $q$-factorial ratios $D_n(q)$ see \cite{ABDJ17}.

\medskip
\noindent
\textbf{Acknowledgements.}
I would like to thank Armin Straub for encouraging me to complete this project and for supply of available knowledge on the topic.
I am grateful to one of the referees whose feedback was terrific and helped me improving the exposition.
Further, I~thank Victor Guo for valuable comments on parts of this work.


\begin{thebibliography}{99}

\bibitem{ABDJ17}
\textsc{B. Adamczewski}, \textsc{J.\,P. Bell}, \textsc{\"E. Delaygue} and \textsc{F. Jouhet},
Congruences modulo cyclotomic polynomials and algebraic independence for $q$-series,
\emph{S\'em. Lotharingien Combin.} \textbf{78B} (2017), Article \#54, 12~pages.

\bibitem{And99}
\textsc{G.\,E. Andrews},
$q$-analogs of the binomial coefficient congruences of Babbage, Wolstenholme and Glaisher,
\emph{Discrete Math.} \textbf{204} (1999), 15--25.

\bibitem{AAR99}
\textsc{G.\,E. Andrews}, \textsc{R. Askey} and \textsc{R. Roy},
\emph{Special functions},
Encyclopedia Math. Appl. \textbf{71} (Cambridge University Press, Cambridge, 1999).

\bibitem{Gor18}
\textsc{O. Gorodetsky},
$q$-congruences, with applications to supercongruences and the cyclic sieving phenomenon,
\emph{Intern. J. Number Theory} (to appear);
\emph{Preprint} \href{http://arxiv.org/abs/1805.01254}{\texttt{arXiv:\,1805.01254 [math.NT]}}.

\bibitem{GZ19}
\textsc{V.\,J.\,W. Guo} and \textsc{W. Zudilin},
A $q$-microscope for supercongruences,
\emph{Adv. in Math.} \textbf{346} (2019), 329--358.

\bibitem{Mes11}
\textsc{R. Me\v strovi\'c},
Wolstenholme's theorem: its generalizations and extensions in the last hundred and fifty years (1862--2012),
\emph{Preprint} \href{http://arxiv.org/abs/1111.3057}{\texttt{arXiv:\,1111.3057 [math.NT]}}.

\bibitem{Pan13}
\textsc{H. Pan},
Factors of some lacunary $q$-binomial sums,
\emph{Monatsh. Math.} \textbf{172} (2013), 387--398.

\bibitem{Str11}
\textsc{A. Straub},
$q$-analog of Ljunggren's binomial congruence,
in \emph{DMTCS Proceedings: 23rd International Conference on Formal Power Series and Algebraic Combinatorics}
(FPSAC  2011),  pp.~897--902 (June 2011);
\emph{Preprint} \href{http://arxiv.org/abs/1103.3258}{\texttt{arXiv:\,1103.3258 [math.NT]}}.

\bibitem{Str18}
\textsc{A. Straub},
Supercongruences for polynomial analogs of the Ap\'ery numbers,
\emph{Proc. Amer. Math. Soc.} \textbf{147} (2019), 1023--1036.

\bibitem{WZ11}
\textsc{A.\,O. Warnaar} and \textsc{W. Zudilin},
A $q$-rious positivity,
\emph{Aequat. Math.} \textbf{81} (2011), 177--183.

\end{thebibliography}
\end{document}